\pdfoutput=1

\documentclass[12pt,a4paper,reqno]{amsart} 
\usepackage{amsmath,amsthm,amssymb,mathrsfs} 
\usepackage{microtype,fixltx2e,lmodern} 
\usepackage{enumerate,comment} 
\usepackage{tikz,float}
\usetikzlibrary{arrows}
\usepackage[utf8]{inputenc} 
\usepackage[T1]{fontenc} 
\usepackage[centering,vscale=0.7,hscale=0.7]{geometry}
\usepackage[hidelinks=true]{hyperref}

\theoremstyle{plain}
\newtheorem{theorem}{Theorem}[section]
\newtheorem{lemma}[theorem]{Lemma}
\newtheorem{proposition}[theorem]{Proposition}
\newtheorem*{conjecture*}{Conjecture}
\newtheorem{corollary}[theorem]{Corollary}
\theoremstyle{definition}
\newtheorem{definition}[theorem]{Definition}
\theoremstyle{remark}
\newtheorem{example}[theorem]{Example}
\newtheorem{remark}[theorem]{Remark}
\numberwithin{equation}{section}

\newcommand{\R}{\mathbb R}
\newcommand{\Z}{\mathbb Z}
\newcommand{\rH}{\mathrm H}
\newcommand{\e}{\mathbf e}

\DeclareMathOperator{\Area}{Area}
\DeclareMathOperator{\rk}{rank}
\DeclareMathOperator{\dist}{dist}

\begin{document}

\title[THE NUMBER OF VERTICES OF A TROPICAL CURVE IS BOUNDED BY ITS AREA]{THE NUMBER OF VERTICES OF A TROPICAL CURVE \\ IS BOUNDED BY ITS AREA}
\author{Tony Yue YU}
\address{Tony Yue YU, Institut de Mathématiques de Jussieu, CNRS-UMR 7586, Case 7012, Université Paris Diderot - Paris 7, Bâtiment Sophie Germain 75205 Paris Cedex 13 France}
\email{yuyuetony@gmail.com}
\date{June 14, 2013 (Revised on December 25, 2014)}

\begin{abstract}
We introduce the notion of tropical area of a tropical curve defined in an open subset of $\mathbb R^n$.
We prove that the number of vertices of a tropical curve is bounded by the area of the curve. The approach is totally elementary yet tricky. Our proof employs ideas from intersection theory in algebraic geometry. The result can be interpreted as the fact that the moduli space of tropical curves with bounded area is of finite type.
\end{abstract}

\maketitle

\section{Introduction and statement of the theorem} \label{sec:intro}

We begin with some heuristic motivations from algebraic geometry.
Let $X$ be a complex projective space.
The moduli space of algebraic curves embedded in $X$ with bounded area with respect to the Fubini-Study metric is of finite type, because the Hilbert schemes are of finite type \cite{Grothendieck_Techniques_1995}.
This article tries to establish an analogous result of finiteness in tropical geometry.
Some combinatorial techniques of this paper are used to study the compactness of tropical moduli spaces in \cite{Yu_Tropicalization_2014} (see also \cite{Gross_Logarithmic_2013,Nishinou_Toric_2006} for related finiteness results).

\begin{theorem}\label{thm: main}
Let $A$ be a positive real number, $U$ an open subset of $\R^n$ for $n\ge 2$, and $K\subset U$ a compact subset.
There exists an integer $N$, such that for any tropical curve $G$ in $U$ with area bounded by $A$, the number of vertices of $G$ inside $K$ is bounded by $N$.
\end{theorem}

Let us explain some of the terminologies used above.

\begin{definition}
Let $\Z/2\Z$ act on $\Z^n\setminus 0$ by multiplication by $-1$, and denote the quotient by $W$. For any $w\in W$, we define its \emph{norm} $|w|=\sqrt{\sum (w^i)^2}$ for some representative $(w^1,...,w^n)\in\Z^n\setminus 0$. We do the same construction for $\mathbb Q^n\setminus 0$, and denote the quotient by $W_{\mathbb Q}$.
\end{definition}

\begin{definition}\label{def: tropical curve}
A \emph{tropical curve} $G$ in an open subset $U\subset\R^n$ is a finite one-dimensional polyhedral complex in $U$ satisfying the following conditions.
\begin{enumerate}[(i)]\itemindent=5mm
\item $G$ is closed in $U$ as a topological subspace. We call the 0-dimensional faces of $G$ \emph{vertices}, and the one-dimensional faces of $G$ \emph{edges}. The set of vertices is denoted by $V(G)$; the set of edges is denoted by $E(G)$. There are two kinds of edges: those edges which have both endpoints in $U$ are called \emph{internal edges}; while the rest are called \emph{unbounded edges}.
\item Each vertex of $G$ is at least 3-valent.
\item Each edge $e$ is equipped with a \emph{weight} $w_e\in W$ parallel to the direction of $e$ inside $\R^n$. If $w_e$ is $k$ times a primitive integral vector, we call $|k|$ the \emph{multiplicity} of the edge $e$.
\item \label{item: balancing condition}We require that the \emph{balancing condition} holds, i.e. for any vertex $v$ of $G$, we have $\sum_{e\ni v} \widetilde w_e = 0$, where the sum is taken over all edges containing $v$ as an endpoint, and $\widetilde w_e$ is the representative of $w_e$ which points outwards from $v$.
\end{enumerate}
\end{definition}

\begin{remark}
The balancing condition in Definition \ref{def: tropical curve}(\ref{item: balancing condition}) is a necessary condition for a tropical curve $G$ to be the amoeba of an analytic curve \cite{Mikhalkin_Enumerative_2005,Nishinou_Toric_2006,Speyer_Uniformizing_2007,Baker_Nonarchimedean_2011}. It is generalized to a global non-toric setting in \cite{Yu_Balancing_2013} using vanishing cycles in $k$-analytic étale cohomology.
\end{remark}

\begin{definition}
For any open subset $V\subset U$, we denote by $G_{|V}$ the restriction of $G$ to $V$.
\end{definition}

\begin{definition}\label{def: area}
For an edge $e$ of a tropical curve $G$, we define its \emph{tropical area} as
\[\Area (e)=|e|\cdot|w_e| ,\]
where $|e|$ means the Euclidean length of the segment $e$ in $\R^n$, and $|w_e|$ is the norm of the weight $w_e$.
The \emph{tropical area} of a tropical curve $G$ is by definition the sum of the tropical areas of all its edges.
In this article, tropical area is simply called \emph{area} for short.
\end{definition}

\begin{example}
Let $e$ be an edge connecting the point $0=(0,\dots,0)$ to the point $x=(x_1,\dots,x_n)$, and let $\widetilde w_e=(w^1,\dots,w^n)\in\Z^n\setminus 0$ be a representative of the weight $w_e$. By definition, there exists $\lambda\in\R$ such that $x=\lambda\cdot \widetilde w_e$. We have
\[\Area(e)=|\lambda| \cdot \sum_{i=1}^n (w^i)^2 .\]
\end{example}

\begin{remark}
To the best of our knowledge, the notion of tropical area in Definition \ref{def: area} did not appear in the existing literature.
It corresponds to the symplectic area under the tropical limit.
There are many ways to see this.
For example, we can explain it in the framework of Berkovich spaces \cite{Berkovich_Spectral_1990}.
Let $k$ be a non-archimedean field with nontrivial valuation, $X$ a closed $k$-analytic annulus of inner radius $r_1$ and outer radius $r_2$, and $f\colon X\rightarrow (\mathbb G_\text{m}^{\text{an}})^n$ a $k$-analytic morphism.
Let $\tau\colon (\mathbb G_\text{m}^{\text{an}})^n\rightarrow\R^n$ denote the tropicalization map taking coordinate-wise valuations.
Suppose that the image $(\tau\circ f)(X)$ is an edge $e$ of the tropical curve $G$.
Put $\omega=\tau^*\big(\sum_{i=1}^n \mathrm d' x_i\wedge \mathrm d'' x_i\big),$
which is a $(1,1)$-form on $(\mathbb G_\text{m}^{\text{an}})^n$ in the sense of \cite{Chambert-Loir_Formes_2012}.
One computes that $\int_X f^*\omega = |w_e|^2\cdot\log\frac{r_2}{r_1}$.
Since $|\omega_e|\cdot\log\frac{r_2}{r_1}=|e|$, we obtain that $\int_X f^*\omega = |w_e|\cdot |e| = \Area(e).$
\end{remark}


Having introduced all the notions, we now explain the proof.
Intuitively, if we regard tropicalization as a classical limit from strings to particles, then the balancing condition resembles a conservation of momentum.
The idea of the proof is to cover our tropical curve by a collection of paths (Section \ref{sec: path}), thought of as paths of particles, and then try to bound the number of vertices on each path (Section \ref{sec: tropical vertex bound}).

We begin by observing that the balancing condition defined locally around each vertex has the following global consequence.

\begin{lemma}\label{lem: global balancing}
Let $G$ be a tropical curve in an open subset $U\subset\R^n$, and let $W$ be an open subset of $U$ such that
\begin{enumerate}[(i)]\itemindent=5mm
\item $\overline W\subset U $;
\item $\overline W$ is a smooth manifold with corners;
\item $V(G)\cap\partial W=\emptyset$;
\item $G$ intersects $\partial W$ transversely.
\end{enumerate}
For each edge $e$ of $G$ that intersects $\partial W$, 
let $\widetilde w_e$ denote the representative of the weight $w_e$ pointing from the inside of $W$ to the outside.
Then we have
\begin{equation}\label{eq: global balancing}
\sum_{e\cap\partial W\neq\emptyset} \widetilde w_e=0 .
\end{equation}
\end{lemma}
\begin{proof}
Let $v_1, \dots, v_l$ be the vertices of $G$ inside $W$, $e_1,\dots,e_m$ the edges of $G$ contained in $W$. Let $B_1,\dots,B_l$ be open balls of radius $r>0$ and with center $v_1,\dots,v_l$. Let $C_1,\dots,C_m$ be open cylinders of radius $r$ and with central axis $e_1,\dots, e_m$. We choose $r$ small enough so that the closures of the balls and the cylinders do not intersect nearby edges and that all of them are contained in $W$. Let $B=\bigcup_{i=1}^l B_i$. We consider a chain of open subsets of $U$ verifying (i)-(iv):
\[B\subset B\cup C_1 \subset B\cup C_1\cup C_2 \subset \dots \subset B\cup C_1\cup\dots\cup C_m \subset W .\]
The equation (\ref{eq: global balancing}) holds for $B$ by the definition of the balancing condition. Then we show by induction that (\ref{eq: global balancing}) holds for every open set in the chain above, and in particular holds for $W$.
\end{proof}

Next, we note that it suffices to prove Theorem \ref{thm: main} in a particular situation.

Let $K'$ be the $n$-simplex obtained as the convex hull of the $n+1$ points $(0,\dots,0)$, $(1,0,\dots,0)$, $\dots$,$(0,\dots,0,1)$ in $\R^n$, and let $U'_\delta$ be the interior of the convex hull of the $n+1$ points $(-\delta,-\delta,\dots,-\delta)$, $(1+3\delta,-\delta,-\delta,\dots,-\delta)$, $(-\delta,1+3\delta,-\delta,\dots,-\delta)$, $\dots$, $(-\delta,\dots,-\delta,1+3\delta)$, where $\delta$ is a positive real number.
Let $U, K$ be the open subset and the compact subset in the statement of Theorem \ref{thm: main}.
For any point $x\in K$, we can find a pair $(U_x, K_x)$ which is isomorphic to $(U'_\delta, K')$ for some $\delta>0$ up to a similarity transformation, such that $x$ is in the interior $K^\circ_x$ of $K_x$ and that $U_x$ is included in $U$.
By the compactness of $K$, there is a finite subset $\{x_1,\dots,x_m\}\subset K$ such that
\[K\subset \bigcup_{i=1}^m K_{x_i}^\circ\subset\bigcup_{i=1}^m K_{x_i}\subset\bigcup_{i=1}^m U_{x_i}\subset U.\]
Therefore, we can deduce Theorem \ref{thm: main} from the following particular situation.

\begin{theorem}\label{thm: main1}
Let $A$ be a positive real number. Let $K', U'_\delta$ be the compact subset and the open subset of $\R^n$ defined as above.
Put $K=K'$, $U=U'_\delta$.
Let $K^\circ$ denote the interior of $K$.
There exists an integer $N$ such that for any tropical curve $G$ in $U$ with area bounded by $A$, the number of vertices of $G_{|K^\circ}$ is bounded by $N$.
\end{theorem}

The proof of Theorem \ref{thm: main1} consists of two parts. The first part (Sections \ref{sec: interpretation of area} - \ref{sec: tropical vertex bound}) treats the case where we have a nice interpretation of the area of a tropical curve as intersection numbers; the second part (Sections \ref{sec: bound on the weights by area} - \ref{sec: saturation trick}) explains how to reduce the general case to the case considered in the first part via a certain modification.
In Section \ref{sec:example}, we give an example to better illustrate Theorem \ref{thm: main}.

\bigskip
\paragraph{\textbf{Acknowledgments}} I am very grateful to Maxim Kontsevich, Bernhard Keller, Antoine Chambert-Loir and Olivier Debarre for discussions and comments.

\section{Interpretation of the area as intersection numbers}\label{sec: interpretation of area}

Let $K$ be as in Theorem \ref{thm: main1}.
In this section, we study a particular type of tropical curves in $K^\circ$, called saturated tropical curves.
We prove that in this case, the area is equal to certain intersection numbers.

The boundary $\partial K$ of $K$ is a simplicial complex of dimension $n-1$. We denote by $(\partial K)^{n-2}$ its skeleton of dimension $n-2$.

\begin{definition}\label{def: saturated}
A tropical curve $G$ in $K^\circ$ is said to be \emph{saturated} if $\overline{G}\cap(\partial K)^{n-2}=\emptyset$ and if $\overline{G}$ intersects $\partial K\setminus(\partial K)^{n-2}$ perpendicularly, where $\overline{G}$ denotes the closure of $G$ in $\R^n$ as a topological subspace.
\end{definition}

\begin{remark}
The word ``saturated'' is used because in this case, the area is concentrated in $K$ in some sense.
\end{remark}

Now let $G$ be a saturated tropical curve in $K^\circ$.
For an intersection point between $\overline G$ and $\partial K$, we define its \emph{multiplicity} to be the multiplicity of the corresponding edge of $G$.

\begin{proposition}\label{prop: d}
The balancing condition implies that $\overline{G}$ intersects each face of $\partial K$ by the same number of times (counted with multiplicity as defined above), which we denote by $d$.
\end{proposition}
\begin{proof}
We use Lemma \ref{lem: global balancing}, where we take $U$ to be $K^\circ$ and \[W=\big\{x\in K^\circ \, \big|\, \dist(x,\partial K)>\epsilon\big\}\] for $\epsilon$ a positive number sufficiently small such that $(U\setminus W)\cap V(G) = \emptyset$. For $1\leq i\leq n$, let $d_i$ be the number of intersections (counting with multiplicity) between $\overline{G}$ and the face of $K$ defined by $x_i=0$. Let $d$ be the number of intersections (counting with multiplicity) between $\overline{G}$ and the face of $K$ defined by $x_1+\dots+x_n=1$. Then equation (\ref{eq: global balancing}) means that
\[d_1 \e_1 + d_2 \e_2 + \dots + d_n \e_n = d(\e_1+\dots+\e_n) ,\]
where we denote by $\e_1,\dots, \e_n$ the vectors in $\R^n$ with coordinates $(1,0,\dots,0)$, $\dots$, $(0,\dots,0,1)$ respectively. Therefore we obtain that $d_1=d_2=\dots =d_n=d$.
\end{proof}

\begin{proposition}\label{prop: area}
$\Area (G)=d$.
\end{proposition}
\begin{proof}
Let $\e_1,\dots, \e_n$ denote the points in $\R^n$ with coordinates $(1,0,\dots,0)$, $\dots$, $(0,\dots,0,1)$ respectively.
Let $K^1$ be the union of the $n$ segments connecting 0 and $\e_i$, for $i\in\{1,\dots,n\}$. We define a measure $\mu$ on $K^1$. We start with the zero measure on $K^1$. For each edge $e$ of $G$, we add to $\mu$ a measure $\mu_e$ defined as follows.
Let $(\alpha_1,\dots,\alpha_n), (\beta_1,\dots,\beta_n)\in K$ be the two endpoints of $e$ and let $(w^1,\dots,w^n)\in\Z^n\setminus\{0\}$ be a representative of the weight of $e$. We define the restriction of $\mu_e$ to the segment connecting 0 and $\e_i$ to be $\mathbf 1_{[\alpha_i,\beta_i]}\cdot |w^i|\cdot \nu$, where $\mathbf 1_{[\alpha_i,\beta_i]}$ is the characteristic function of the segment $[\alpha_i,\beta_i]\subset \R$, and $\nu$ denotes the one-dimensional Lebesgue measure. Then by Definition \ref{def: area}, the area of $G$ is the total mass of $\mu$. Let us now calculate the measure $\mu$.

Let $z^{(1)},\dots, z^{(l)}$ be the intersection points between $\overline{G}$ and the face of $K$ defined by $x_1+\dots +x_n=1$ with multiplicity $m^{(1)}, \dots, m^{(l)}$ respectively. We have $m^{(1)}+\dots+m^{(l)}=d$ by Proposition \ref{prop: d}. Let $\big(z_1^{(k)},\dots,z_n^{(k)}\big)$ be the coordinates of $z^{(k)}$ for $k=1,\dots,l$. We fix $i\in\{1,\dots,n\}$ and assume that $z_i^{(1)}\leq z_i^{(2)} \leq \dots\leq z_i^{(l)}$. Let $\mu_i$ denote the restriction of $\mu$ to the segment connecting 0 and $\e_i$.

\begin{lemma}
We have
\[\mu_i = \sum_{k=1}^l m^{(k)}\cdot \mathbf 1_{[0, z_i^{(k)}]} \quad\text{almost everywhere.}\]
\end{lemma}
\begin{proof}
Let $z_i^{(0)}=0$, $z_i^{(l+1)}=1$, and $\zeta\in (0,1)$. Assume that there is no vertex of $G$ with $i^\text{th}$ coordinate equal to $\zeta$ and that $z_i^{(j)} < \zeta < z_i^{(j+1)}$, for some $j\in\{0,\dots, l\}$.
Let us show that the density of $\mu_i$ at the point $\zeta\cdot\mathbf e_i$ is $d-\sum_{k=1}^j m^{(k)}$, which we denote by $d_\zeta$. Let $H_\zeta^-$ be the half space $\big\{ (x_1,\dots,x_n)\in\R^n\,\big|\, x_i\leq\zeta\big\}$, $W$ the interior of $K\cap H_\zeta^-$.
Lemma \ref{lem: global balancing} implies that
\[\sum_{p\in G\cap\partial H^-_\zeta} \big|w_{e(p)}^i\big|=d_\zeta,\]
where $e(p)$ denotes the edge of $G$ containing $p$.
So by construction, the tropical curve $G$ contributes $d_\zeta$ to the density of $\mu_i$ at the point $\zeta\cdot\mathbf e_i\in [0,\mathbf e_i]$.
\end{proof}

We continue the proof of Proposition \ref{prop: area}. We calculate the total mass of $\mu$, denoted by $m(\mu)$. We have
\[m(\mu) = \sum_{i=1}^{n} m(\mu_i) = \sum_{i=1}^n\sum_{k=1}^l m^{(k)}\cdot z_i^{(k)} = \sum_{k=1}^l m^{(k)} \sum_{i=1}^n z_i^{(k)} = \sum_{k=1}^l m^{(k)} = d ,\]
completing the proof of the proposition.
\end{proof}

\section{Paths and collection of paths}\label{sec: path}

In this section, we introduce the notion of paths and collection of paths.

Let $R$ be an n-dimensional polyhedron in $\R^n$, $V$ an open subset of $\R^n$ containing $R$.
We fix a direction $i\in\{1,\dots,n\}$.
We assume that $R$ has an $(n-1)$-dimensional face $F$ contained in a hypersurface defined by $x_i=c$, for some $c\in\R$, and that $R$ is contained in the half space $x_i\geq c$. Morally, we can think of the $i^\text{th}$ direction as time, and the rest as space directions. Let $H$ be a tropical curve in $V$ such that there is an edge $e_0$ of $H$ whose interior intersects the relative interior of $F$ transversely.

\begin{definition}\label{def: path}
A \emph{path} $P$ starting from $e_0$ with direction $i$ is a chain of weighted segments $s_0, s_1,\dots,s_{l_P}$ satisfying the following conditions.
\begin{enumerate}[(i)]\itemindent=5mm
\item $s_0=e_0\cap R$, $s_{l_P}=e'_0\cap R$ for some edge $e'_0$ of $H$ such that exactly one endpoint of $e'_0$ does not belong to the interior $R^\circ$.
\item $s_1,\dots,s_{l_P-1}$ are edges of $H$, and $s_1,\dots,s_{l_P-1}\subset R^\circ$.
\item Every two consecutive segments in the chain share one endpoint.
\item \label{def: path: injectivity}The projection to the $i^\text{th}$ coordinate $\R^n\rightarrow \R$ restricted to $P$ is injective.
\item Each segment $s_j$ carries the weight $w'_{s_j} = w_e/|w^i_e|\in W_{\mathbb Q}$, where $e$ is the edge of $H$ containing $s_j$.
\end{enumerate}
\end{definition}

\begin{definition}
A \emph{union} $U$ of $m$ paths $P_1,\dots,P_m$ is a polyhedral sub-complex of $H$, such that
\begin{enumerate}[(i)]\itemindent=5mm
\item Set theoretically $U=\bigcup_{j=1}^m P_j$;
\item Each segment $s$ of $U$ carries the weight $w'_s = k\cdot w_e/|w^{i}_e|\in W_{\mathbb Q}$, where $e$ is the edge of $H$ containing $s$, and $k=\#\big\{j \,\big|\, P_j \text{ contains }s\big\}$.
\end{enumerate}
\end{definition}

\begin{lemma}\label{lem: collection of paths}
Let $m=|w^i_{e_0}|$.
There exists a collection of $m$ paths $P_1,\dots, P_m$ starting from $e_0$ with direction $i$ such that each segment $s$ in the union $U=\bigcup_{j=1}^m P_j$ verifies the following property: 

Let $e$ be the edge of $H$ containing $s$, and let $\widetilde w'_s$ and $\widetilde w_e$ be representatives of the weights $w'_s$ and $w_e$ respectively. By construction $\widetilde w'_s$ and $\widetilde w_e$ are parallel, so there exists $q\in\mathbb Q$ such that $\widetilde w'_s = q \widetilde w_e$. The property is that $|q|\leq 1$. 
\end{lemma}
\begin{proof}
We assign to each edge $e$ of our tropical curve $H$ an integer $c_i(e)$ called capacity (in the $i^\text{th}$ direction). Initially we set $c_i(e)=|w_e^i|$. To construct the path $P_1$, we start with the segment $s_0=e_0\cap R$, and we decrease the capacity $c_i(e_0)$ by 1. Suppose we have constructed a chain of segments $s_0, s_1,\dots,s_j$. Let $B$ be the endpoint of $s_j$ with larger $i^\text{th}$ coordinate. If $B\in\partial R$ we stop, otherwise we choose $e_{j+1}$ to be an edge of $H$ such that:
\begin{enumerate}[(i)]\itemindent=5mm
\item $B$ is an endpoint of $e_{j+1}$.
\item For any point $x\in e_{j+1}\setminus B$, the $i^\text{th}$ coordinate of $x$ is larger than the $i^\text{th}$ coordinate of $B$.
\item The capacity $c_i(e_{j+1})$ is positive.
\end{enumerate}
The existence of such $e_{j+1}$ is ensured by the balancing condition on $H$. After choosing $e_{j+1}$, we decrease the capacity $c_i(e_{j+1})$ by 1 and set $s_{j+1}=e_{j+1}\cap R$. We iterate this procedure until we stop, and we obtain the path $P_1$. We apply the same procedure $m$ times and obtain the collection of paths $P_1, \dots, P_m$ as required in the lemma.
\end{proof}

\section{Tropical vertex bound and genus bound}\label{sec: tropical vertex bound}

Let $K$ be as in Theorem \ref{thm: main1}, and let $G$ be a saturated tropical curve in $K^\circ$ with area $d$ as in Section \ref{sec: interpretation of area}. In this section, we give a very coarse bound on the number of vertices of $G$ in terms of the area $d$ and the dimension $n$.

\begin{proposition}\label{prop: vertex bound}
$\#V(G)\leq 2(n-1)^2 d^2 .$
\end{proposition}
\begin{proof}
Let $(x_1,\dots,x_n)$ be the standard coordinates on $\R^n$. We fix a direction $i\in\{1,\dots,n\}$. Let $z_i^{(1)},\dots, z_i^{(l)}$ be the intersection points between $\overline{G}$ and the face of $K$ defined by $x_i=0$ with multiplicity $m^{(1)},\dots, m^{(l)}$ respectively. By Proposition \ref{prop: area}, we have $m^{(1)}+\dots+m^{(l)} = d$. Let $e_{0i}^{(k)}$ be the edge of $G$ corresponding to the intersection point $z_i^{(k)}$. For each intersection point $z_i^{(k)}$, by Lemma \ref{lem: collection of paths}, we obtain a collection of $m^{(k)}$ paths starting from $e_{0i}^{(k)}$ with direction $i$. So for $k=1,\dots,l$, we obtain in total $d$ paths, and we label them as $P_{i1},\dots,P_{id}$. For each such path $P$, let $V(P)$ denote the set of vertices of $P$ that lies in $K^\circ$, and let $V_0(P)$ be the following subset of $V(P)$.

A vertex $Q$ belongs to $V_0(P)$ if and only if there is an edge of $G$, denoted by $e(Q)$, such that
\begin{enumerate}[(i)]\itemindent=5mm
\item The vertex $Q$ is an endpoint of the edge $e(Q)$.
\item The edge $e(Q)$ is not in contained in the path $P$.
\item There exists $j\in\{1,\dots,n\}$, $j\neq i$, such that the $j^\text{th}$ component of $w_{e(Q)}$ is non-zero.
\end{enumerate}

We claim that (see Lemma \ref{lem: bound on one path})
\begin{equation}\label{eq: bound on one path}
\# V_0(P)\leq 2d(n-1) .
\end{equation}
Now we vary $i$, and in the same way, we get $nd$ paths $P_{ik}$ for $i=1,\dots,n$, $k=1,\dots,d$. We claim that (see Lemma \ref{lem: cover by paths})
\begin{equation}\label{eq: cover by paths}
\bigcup_{i=1}^{n-1} \bigcup_{k=1}^d V_0(P_{ik}) \supset V(G) .
\end{equation}
Combining equations (\ref{eq: bound on one path}) and (\ref{eq: cover by paths}), we have proved our proposition.
\end{proof}

\begin{lemma}\label{lem: bound on one path}
For a path $P$ among the paths $P_{ik}$ constructed in the proof above, we have the following bound
\[\#V_0(P)\leq 2d(n-1) .\]
\end{lemma}
\begin{proof}
Let $S_{P,j} =\sum_{Q\in V_0(P)} |w_{e(Q)}^j|$ for $j\in\{1,\dots,\widehat i,\dots,n\}:=\{1,\dots,n\}\setminus\{i\}$, where $e(Q)$ is the edge of $G$ associated to the vertex $Q$ as in the definition of $V_0(P)$ in the proof of Proposition \ref{prop: vertex bound}. Now we fix $j$, and let 
\[E^-_{0,j} (P) =\big\{e(Q)\, \big|\, Q\in V_0(P), \widetilde w_{e(Q)}^j <0\big\} ,\]
\[E^+_{0,j} (P) =\big\{e(Q)\, |\, Q\in V_0(P), \widetilde w_{e(Q)}^j >0\big\} ,\]
where $\widetilde w_{e(Q)}$ is the representative of $w_{e(Q)}$ that points outwards from $Q$. Let
\[S^-_{P,j} = \sum_{e\in E^-_{0,j}(P)} -\widetilde w_{e}^{j} ,\]
\[S^+_{P,j} = \sum_{e\in E^+_{0,j}(P)} \widetilde w_{e}^{j} .\]

Let $p_i\colon\R^n\rightarrow\R$ be the projection to the $i^\text{th}$ coordinate, and $p_j$ the projection to the $j^\text{th}$ coordinate. By Definition \ref{def: path}(\ref{def: path: injectivity}), ${p_i}_{|P}$ is injective. Assume that the image of ${p_i}_{|P}$ is the closed interval $[0,z_i^{P}]$. Let
\[q_i=
\begin{cases}
({p_i}_{|P})^{-1}(0) &\text{for }x_i\in (-\infty,0]\\
({p_i}_{|P})^{-1}(x_i) &\text{for }x_i\in [0,z_i^P]\\
({p_i}_{|P})^{-1}(z_i^P) &\text{for }x_i\in [z_i^P,\infty)
\end{cases}\]
\[R_{j-}=\big\{(x_1,\dots,x_n)\in\R^n \, \big| \, x_j\leq p_j(q_i(x_i))-\epsilon\big\} .\]
We choose $\epsilon$ to be a sufficiently small positive real number such that
\begin{enumerate}[(i)]\itemindent=5mm
\item $R^\circ_{j-}\supset\big\{(x_1,\dots,x_n)\in\R^n \, \big|\, x_j\leq 0\big\}$.
\item $\partial R_{j-}\cap V(\overline G)=\emptyset$.
\item $\partial R_{j-}$ intersects $\overline G$ transversely.
\item $\forall e\in E^-_{0,j} (P)$, $e\cap R_{j-}^\circ\neq\emptyset$.
\end{enumerate}
Let
\[T_{j-}=\partial (R_{j-}\cap K)\setminus (\partial K \cap\{x_j=0\}) .\]
Then for any $y\in T_{j-}\cap \overline{G}$, let $e(y)$ denote the edge of $G$ corresponding to the intersection point $y$. By Lemma \ref{lem: global balancing}, we have \[\sum_{y\in T_{j-}\cap \overline{G}} |w_{e(y)}^j|=d .\]
Therefore $S^-_{P,j}\leq d$, and similarly $S^+_{P,j}\leq d$, so $S_{P,j}=S_{P,j}^-+S_{P,j}^+\leq 2d$. Let $S_P=\sum_{1\leq j \leq n, j\neq i} S_{P,j}$. We have $S_P\leq 2d(n-1)$. By the definition of the set $V_0(P)$, each vertex $Q\in V_0(P)$ contribute at least 1 to the quantity $S_P$ so we obtain that $\#V_0(P)\leq 2d(n-1)$.
\end{proof}

\begin{lemma}\label{lem: cover by paths}
Let $P_{ik}, V_0$ be as in the proof of Proposition \ref{prop: vertex bound}, we have
\[\bigcup_{i=1}^{n-1} \bigcup_{k=1}^d V_0(P_{ik}) \supset V(G) .\]
\end{lemma}
\begin{proof}
By Lemma \ref{lem: global balancing} and Lemma \ref{lem: collection of paths}, we see that for any edge $e\subset G$, any $i\in\{1,\dots,n\}$ such that $w_e^i\neq 0$, there exists $k\in\{1,\dots,d\}$ such that the path $P_{ik}$ constructed in the proof of Proposition \ref{prop: vertex bound} contains $e$. Now for any vertex $v$ of $G$, since $v$ is at least 3-valent by definition, there exists an edge $e$ of $G$ containing $v$ such that $w_e^i\neq 0$ for some $i\in\{1,\dots, n-1\}$. This means that there exists $k\in\{1,\dots,d\}$ such that the path $P_{ik}$ contains $e$ by what we have just said. However it can happen that $v\notin V_0(P_{ik})$. In such cases, by the definition of the set $V_0(P_{ik})$, there exists another edge $e'\not\subset P_{ik}$ such that $w_{e'}^j=0$ for any $j\in\{1,\dots,\widehat i,\dots, n\}$. Since $w^i_{e'}\neq 0$, there exists $k'\in\{1,\dots,d\}$ such that the path $P_{ik'}$ contains $e'$. Since $e'\neq e$, there exists $j\in\{1,\dots,\widehat i,\dots,n\}$ such that $w^j_e\neq 0$, which implies that $v\in V_0(P_{ik'})$. To sum up, we have proved that for any vertex $v$ of $G$, there exists $i\in\{1,\dots,n-1\}$, $k\in\{1,\dots,d\}$ such that $v\in V_0(P_{ik})$, so we have proved our lemma.
\end{proof}

\begin{remark}
By analogy with algebraic geometry, we can expect a much better bound on the number of vertices based on the Castelnuovo bound on the genus of a smooth curve of given degree in the projective space $\mathbf P^n$ (see for example \cite{Arbarello_Geometry_1985}). Indeed, once we know how to bound the genus of our tropical curve $G$, which is by definition $\rk \rH_1 (G)$, we can bound the number of vertices immediately. For example, using cellular homology to calculate the Euler characteristic of $G$, we have
\[1 - \rk \rH_1(G) = \# V(G) - \#\{\text{internal edges of $G$}\} .\]
Then it suffices to observe that the number of internal edges is bounded below by the hypothesis that each vertex is at least 3-valent.
\end{remark}

\begin{conjecture*}
The number of vertices of $G$ is bounded by $2 \pi(d,n) + (n+1)d -2$, where $\pi(d,n)$ is defined by
\[\pi(d,n)=\frac{m(m-1)}{2}(n-1)+m\epsilon ,\]
\[\text{ where }m=\left[\frac{d-1}{n-1}\right] \text{ and } \epsilon = d-1-m(n-1) .\]
When $d>2n$, the bound may be achieved by a tropical analogue of Castelnuovo curves.
\end{conjecture*}

\section{Bound on the weights by the area}\label{sec: bound on the weights by area}

In this section, we show that the weights of the edges of a tropical curve can be bounded by the area.

\begin{proposition}\label{prop: flow}
Fix $i\in\{1,\dots,n\}$. Let $R$ be the convex hull of the $2^n$ points
\[\Big\{(\epsilon_1,\dots,\epsilon_n)\in\R^n\,\Big|\, \epsilon_j\in\{-1,+1\}\text{ for }j\in\{1,\dots,\widehat i,\dots,n\}, \ \epsilon_i\in\{0,1\}\Big\} .\]
Let $V$ be an open set in $\R^n$ containing $R$. Let $H$ be a tropical curve in $V$ such that there is an edge $e_0$ of $H$ whose interior contains the point $0=(0,\dots,0)\in\R^n$. Then we have
\[\Area (H_{|R^\circ}) \geq |w_{e_0}^i| ,\]
where $w_{e_0}^i$ denotes the $i^\text{th}$ component of the weight of the edge $e_0$.
\end{proposition}
\begin{proof}
Denote $m=|w_{e_0}^i|$. By Lemma \ref{lem: collection of paths}, we obtain a collection of $m$ paths $P_1,\dots,P_m$ starting from $e_0$ with direction $i$. Each path $P_k$ connects the origin $O$ with a point on the boundary $\partial R$, denoted by $z_k$. By Definition \ref{def: path} (\ref{def: path: injectivity}), the $i^\text{th}$ coordinate of $z_k$ is strictly positive. This implies in particular that the length of $P_k$ under the Euclidean metric is at least one, so we have $\Area (P_k)\geq 1$. By summing up contributions from all $P_k$, for $k=1,\dots,m$, we obtain that $\Area (H_{|R^\circ}) \geq m$.
\end{proof}

\begin{corollary}\label{cor: bound on intersection nb}
Let $A,U,K,G,\delta$ be as in Theorem \ref{thm: main1}, and denote by $I$ the number of intersection points between $G$ and $\partial K$ (with no multiplicity concerned). Then $I\leq A/\delta$.
\end{corollary}
\begin{proof}
By Proposition \ref{prop: flow}, each intersection point contributes at least $\delta$ to the total area of $G$, whence the corollary.
\end{proof}

\begin{corollary}\label{cor: bound on weight}
Let $A,U,K,G,\delta$ be as in Theorem \ref{thm: main1}. For any edge $e$ of $G_{|K^\circ}$, any $i\in\{1,\dots,n\}$, we have $| w_e^i | \leq A/\delta$.
\end{corollary}
\begin{proof}
By Proposition \ref{prop: flow}, for any edge $e$ of $G_{|K^\circ}$, any $i\in\{1,\dots,n\}$, the weight $w_e$ contributes at least $|w^i_e|\cdot\delta$ to the total area of $G$, whence the corollary.
\end{proof}

\section{The saturation trick}\label{sec: saturation trick}

Finally we perform a trick to reduce the general case to the saturated case considered in Sections \ref{sec: interpretation of area} and \ref{sec: tropical vertex bound}.
Using the notations and the assumptions as in Theorem \ref{thm: main1}, our aim is to construct from $G$ a saturated tropical curve $G'$ in $K^\circ$ (in the sense of Definition \ref{def: saturated}).

Let $\epsilon$ be a positive real number and put
\[\check K=\big\{x\in K^\circ\,\big |\, \dist(x,\partial K)> \epsilon\big\} .\]
We choose $\epsilon$ small enough such that $V(G)\cap (K^\circ \setminus \check K) = \emptyset$.

Let $\e'_1,\dots, \e'_n$ denote the vectors in $\R^n$ with coordinates $(-1,0,\dots,0)$, $\dots$, $(0,\dots,0,-1)$ respectively.
Let $\e'_0=-(\e'_1+\dots+\e'_n)=(1,1,\dots,1)$.
The following lemma is obvious.

\begin{lemma}\label{lem:coordinates}
For any $w\in\Z^n$, there exists a unique collection of non-negative integers $a_0,\dots,a_n$ such that
\[w=\sum_{i=0}^n a_i \e'_i ,\]
and that $a_i$ is zero for at least one $i\in\{0,\dots,n\}$.
\end{lemma}

Initially we set $G'=G_{|K^\circ}$. Then for each edge $e$ of $G'$ such that the closure $\overline e$ intersects $\partial K$ non-perpendicularly, or $\overline e\cap(\partial K)^{n-2}\neq\emptyset$, we do the following modification to $G'$. Let $w_e$ be the weight of $e$ and choose the representative $\widetilde w_e$ that points from $\check K$ to $K^\circ\setminus\check K$. Now put $\widetilde w_e$ into Lemma \ref{lem:coordinates} and we get $(n+1)$ non-negative integers $a_0,\dots,a_n$. Let $P=e\cap \partial \check K$, $\widehat e=(K^\circ\setminus \overline{\check K})\cap e$. We first delete $\widehat e$ from $G'$. Now P becomes an unbalanced vertex. Then we add to $G'$ the rays starting from $P$ with direction $\e'_i$ and multiplicity $a_i$ for all $i\in\{0,\dots,n\}$. This makes the vertex $P$ balanced again and we finish our modification concerning the edge $e$ (see Figure \ref{fig:saturation_trick}).

\begin{figure}[H]
\centering
\begin{tikzpicture}[line cap=round, line join=round, >=triangle 45, x=1.0cm, y=1.0cm,scale=0.6]
\clip(-1,-0.1) rectangle (13,4.1);
\draw (0,4)-- (0,0);
\draw (0,0)-- (4,0);
\draw (4,0)-- (0,4);
\draw (8,0)-- (12,0);
\draw (12,0)-- (8,4);
\draw (8,4)-- (8,0);
\draw (9.66510998574419,0.6674450071279052)-- (10.6,0.2);
\draw (1.6,0.7)-- (3.0,0);
\draw [dash pattern=on 2pt off 2pt] (1.0,1.0)-- (1.6,0.7);
\draw [dash pattern=on 2pt off 2pt] (9.0,1.0)-- (9.66510998574419,0.6674450071279052);
\draw (10.6,0.2)-- (10.6,0);
\draw (10.6,0.2)-- (11.2,0.8);
\draw (2.2,0.85) node[anchor=north west] {$e$};
\draw [->,line width=0.6pt] (4.5,2.0) -- (6.5,2.0);
\end{tikzpicture}
\caption{}\label{fig:saturation_trick}
\end{figure}

\begin{lemma} \label{lem:saturation_trick}
Using the notations in Theorem \ref{thm: main1}. By construction we have the following.
\begin{enumerate}[(i)]\itemindent=5mm
\item $G'$ is a saturated tropical curve in $K^\circ$.
\item $\# V(G_{|K^\circ}) \leq \# V(G')$.
\item $\Area(G') \leq \Area(G_{|K^\circ})+n(A/\delta)^2 \leq A+n(A/\delta)^2$.
\end{enumerate}
\end{lemma}
\begin{proof}
(i) follows directly from the construction. (ii) is obvious since our modification may add new vertices to $G_{|K^\circ}$ but never decreases the number of vertices.
For (iii), each time we do a modification to an edge, we add at most $n$ rays, each of which has area less than $A/\delta$ (Corollary \ref{cor: bound on weight}). Moreover by Corollary \ref{cor: bound on intersection nb}, there are at most $A/\delta$ edges of $G$ intersecting with $\partial K$, so the total area of all the rays we added to $G_{|K^\circ}$ is bounded by $n (A/\delta)^2$.
\end{proof}

To conclude, combining Lemma \ref{lem:saturation_trick} with Proposition \ref{prop: vertex bound}, we have proved Theorem \ref{thm: main1}, with $2 (n-1)^2 (A+n (A/\delta)^2)^2$ being the bound on the number of vertices.
We have also proved Theorem \ref{thm: main} using the reduction explained at the end of Section \ref{sec:intro}.

\section{An example of a tropical curve with finite area but infinite number of vertices} \label{sec:example}

To better illustrate Theorem \ref{thm: main}, we give an example of a tropical curve\footnote{Here we drop the finiteness assumption in Definition \ref{def: tropical curve} of tropical curves.} $G$ in an open subset $U\subset \R^2$ which has finite area $A$ but infinite number of vertices.
It does not contradict Theorem \ref{thm: main} because the number of vertices of $G$ inside any compact subset in $U$ will still be finite.
Intuitively, Theorem \ref{thm: main} says that concentrations of vertices can only happen near the boundary of $U$ as long as the area of the tropical curve is bounded.

Let $(x,y)$ be coordinates on $\R^2$. Let $C$ be the convex hull of the four points $(0,0), (0,1), (1,0), (1,1)$ in $\R^2$ and let $U$ be the interior of $C$.
Our tropical curve $G$ consists of the following segments (they are all taken after intersection with $U$):
\begin{enumerate}[(i)]\itemindent=5mm
\item the segment $[(4^{-n},4^{-n}),(4^{-(n-1)},4^{-(n-1)})]$ with multiplicity $2^{n-1}$,
\item the ray starting at the point $(4^{-n},4^{-n})$ with direction $(-1,0)$ and multiplicity $2^{n-1}$,
\item the ray starting at the point $(4^{-n},4^{-n})$ with direction $(0,-1)$ and multiplicity $2^{n-1}$,
\item the ray starting at the point $(4^{-n},4^{-n})$ with direction $(-1,2)$ and multiplicity $2^{n}$,
\item the ray starting at the point $(4^{-n},4^{-n})$ with direction $(2,-1)$ and multiplicity $2^{n}$,
\end{enumerate}
where $n$ is taken over all positive integers (see Figure \ref{fig:example}).
One checks that the balancing condition is verified (see Definition \ref{def: tropical curve}(\ref{item: balancing condition})).

\begin{figure}[H]
\centering
\begin{tikzpicture}[scale=4]
\draw [gray] (0, 0) rectangle (1, 1);
\draw (0,0) -- (1,1);
\foreach \n in {1,...,10}
{
\draw ({pow(4, -\n)}, {pow(4, -\n)}) -- (0, {pow(4, -\n)});
\draw ({pow(4, -\n)}, {pow(4, -\n)}) -- (0, {pow(4, -\n)*3});
\draw ({pow(4, -\n)}, {pow(4, -\n)}) -- ({pow(4, -\n)},0);
\draw ({pow(4, -\n)}, {pow(4, -\n)}) -- ({pow(4, -\n)*3},0);
}
\end{tikzpicture}
\caption{\label{fig:example}}
\end{figure}

\begin{proposition}
We have $\Area(G) = 14$.
\end{proposition}
\begin{proof}
For each integer $n\geq 1$, the segments from (i)-(v) contribute to $\Area(G)$ by $3\cdot 2^{-n}, \ 2^{-(n+1)}, \ 2^{-(n+1)}, \ 5\cdot 2^{-n}, \ 5\cdot 2^{-n}$ respectively.
Summing over all $n\geq 1$ we get $\Area(G) = 14$.
\end{proof}


\begin{thebibliography}{10}

\bibitem{Arbarello_Geometry_1985}
E.~Arbarello, M.~Cornalba, P.~A. Griffiths, and J.~Harris.
\newblock {\em Geometry of algebraic curves. {V}ol. {I}}, volume 267 of {\em
  Grundlehren der Mathematischen Wissenschaften [Fundamental Principles of
  Mathematical Sciences]}.
\newblock Springer-Verlag, New York, 1985.

\bibitem{Baker_Nonarchimedean_2011}
Matthew Baker, Sam Payne, and Joseph Rabinoff.
\newblock Nonarchimedean geometry, tropicalization, and metrics on curves.
\newblock {\em arXiv preprint arXiv:1104.0320}, 2011.

\bibitem{Berkovich_Spectral_1990}
Vladimir~G. Berkovich.
\newblock {\em Spectral theory and analytic geometry over non-{A}rchimedean
  fields}, volume~33 of {\em Mathematical Surveys and Monographs}.
\newblock American Mathematical Society, Providence, RI, 1990.

\bibitem{Chambert-Loir_Formes_2012}
Antoine Chambert-Loir and Antoine Ducros.
\newblock Formes diff\'erentielles r\'eelles et courants sur les espaces de
  berkovich.
\newblock {\em arXiv preprint arXiv:1204.6277}, 2012.

\bibitem{Gross_Logarithmic_2013}
Mark Gross and Bernd Siebert.
\newblock Logarithmic {G}romov-{W}itten invariants.
\newblock {\em J. Amer. Math. Soc.}, 26(2):451--510, 2013.

\bibitem{Grothendieck_Techniques_1995}
Alexander Grothendieck.
\newblock Techniques de construction et th\'eor\`emes d'existence en
  g\'eom\'etrie alg\'ebrique. {IV}. {L}es sch\'emas de {H}ilbert.
\newblock In {\em S\'eminaire {B}ourbaki, {V}ol.\ 6}, pages Exp.\ No.\ 221,
  249--276. Soc. Math. France, Paris, 1995.

\bibitem{Mikhalkin_Enumerative_2005}
Grigory Mikhalkin.
\newblock Enumerative tropical algebraic geometry in {$\Bbb R^2$}.
\newblock {\em J. Amer. Math. Soc.}, 18(2):313--377, 2005.

\bibitem{Nishinou_Toric_2006}
Takeo Nishinou and Bernd Siebert.
\newblock Toric degenerations of toric varieties and tropical curves.
\newblock {\em Duke Math. J.}, 135(1):1--51, 2006.

\bibitem{Speyer_Uniformizing_2007}
David~E Speyer.
\newblock Uniformizing tropical curves {I}: genus zero and one.
\newblock {\em arXiv preprint arXiv:0711.2677}, 2007.

\bibitem{Yu_Balancing_2013}
Tony~Yue Yu.
\newblock Balancing conditions in global tropical geometry.
\newblock {\em arXiv preprint arXiv:1304.2251}, 2013.

\bibitem{Yu_Tropicalization_2014}
Tony~Yue Yu.
\newblock Tropicalization of the moduli space of stable maps.
\newblock {\em arXiv preprint arXiv:1407.8444}, 2014.

\end{thebibliography}

\end{document}